\documentclass[preprint,12pt]{elsarticle}
\usepackage{graphicx}
\usepackage[cp1251]{inputenc}
\usepackage{amssymb}
\usepackage{amsthm}
\usepackage{cmap}

\usepackage{amsmath}% http://ctan.org/pkg/amsmath

\biboptions{sort&compress}

\newcommand{\sysn}{\left\{\begin{array}{rcl}}
\newcommand{\sysk}{\end{array}\right.}

\newtheorem{theorem}{Theorem}[section]

\theoremstyle{example}

\newtheorem{proposition}[theorem]{Proposition}
\theoremstyle{definition}

\newtheorem{remark}[theorem]{Remark}
\newtheorem{corollary}[theorem]{Corollary}

\journal{...}

\begin{document}

\title{The $\kappa$-Fr\'{e}chet-Urysohn property for $C_p(X)$ is equivalent to Baireness of $B_1(X)$}

\author{Alexander V. Osipov}

\address{Krasovskii Institute of Mathematics and Mechanics, \\ Ural Federal
 University, Yekaterinburg, Russia}

\ead{OAB@list.ru}

\begin{abstract} A topological space $X$ is {\it Baire} if the intersection of any
sequence of open dense subsets of $X$ is dense in $X$.

We establish that the property $(\kappa)$ for a Tychonoff space $X$ is equivalent to Baireness of $B_1(X)$ and, hence,  the Banakh property for $C_p(X)$ is equivalent to  meagerness of $B_1(X)$.
Thus, we obtain one characteristic of the Banakh property for $C_p(X)$  through the property of the space $X$.
%Finally, we proved that a pseudocompact space $X$ has the property $(\kappa)$ if and only if all countable subsets of $X$ are scattered.  This answers a question of Tkachuk.
\end{abstract}
%\tnotetext[label1]{The research has been supported by .}


\maketitle %%
%% Start line numbering here if you want
%%
% \linenumbers

%% main text

\section{Introduction}

A space $X$ is {\it Fr\'{e}chet-Urysohn} if, for any $A\subseteq X$ and any $x\in A$, there exists
a sequence $\{a_n: n\in \mathbb{N}\}\subseteq A$ that converges to $x$.

A space $X$ is called {\it $k$-Fr\'{e}chet-Urysohn} if, for any open set $U\subset X$ and any point $x\in \overline{U}$, there exists a sequence $\{x_n: n\in \mathbb{N}\}\subset U$ that converges to~$x$. Clearly, every Fr\'{e}chet-Urysohn space is $k$-Fr\'{e}chet-Urysohn.

The class of $k$-Fr\'{e}chet-Urysohn spaces was introduced by Arhangel'skii but, long before that, Mr\'{o}wka proved in \cite{Mrowka} (without using the term) that any product of first countable spaces is $k$-Fr\'{e}chet-Urysohn.

 In \cite{Sakai}, Sakai characterized $\kappa$-Fr\'{e}chet-Urysohn property in $C_p(X)$ showing that $C_p(X)$ is $\kappa$-Fr\'{e}chet-Urysohn if and only if $X$ has the property $(\kappa)$.

A space $X$ is said to have {\it property $(\kappa)$} if every pairwise
disjoint sequence of finite subsets of $X$ has a strongly point-finite subsequence.

A family $\{A_{\alpha}: \alpha\in \kappa\}$ of subsets of a set $X$ is said to be {\it point-finite} if for every $x\in X$, $\{\alpha\in \kappa: x\in A_{\alpha}\}$ is finite.

A family $\{A_{\alpha}: \alpha\in \kappa\}$ of subsets of a space $X$ is said to be {\it strongly point-finite} if for every $\alpha\in \kappa$, there exists an open
set $U_{\alpha}$ of $X$ such that $A_{\alpha}\subset U_{\alpha}$ and $\{U_{\alpha}: \alpha\in \kappa\}$ is point-finite.

\medskip

A topological space $X$ is {\it Baire} if the intersection of any
sequence of open dense subsets of $X$ is dense in $X$.

 We establish that the $\kappa$-Fr\'{e}chet-Urysohn property for $C_p(X)$ is equivalent to Baireness of $B_1(X)$.

\section{Main definitions and notation}

In this paper all spaces are assumed to be Tychonoff. The expression $C(X,Y)$ denotes the set of all continuous maps from a space $X$
to a space $Y$. We follow the usual practice to write $C(X)$ instead of $C(X,\mathbb{R})$. The
space $C_p(X)$ is the set $C(X)$ endowed with the pointwise convergence topology.

A real-valued function $f$ on a space $X$ is a {\it Baire-one
function} (or a {\it function of the first Baire class}) if $f$ is
a pointwise limit of a sequence of continuous functions on $X$.
Let $B_1(X)$ denote the space of all Baire-one real-valued
functions on a space  $X$ with the topology of pointwise convergence.

 We recall that a subset of $X$ that is the
 complete preimage of zero for a certain function from~$C(X)$ is called a zero-set.
A subset $O\subseteq X$  is called  a cozero-set (or functionally
open) of $X$ if $X\setminus O$ is a zero-set of $X$. It is easy to
check that zero sets are preserved by finite unions and countable
intersections. Hence cozero sets are preserved by finite
intersections and countable unions. Countable unions of zero sets
will be denoted by $Zer_{\sigma}$ (or $Zer_{\sigma}(X)$),
countable intersection of cozero sets by $Coz_{\delta}$ (or
$Coz_{\delta}(X)$). It is easy to check that $Zer_{\sigma}$-sets
are preserved by countable unions and finite intersections.

Note that any zero-set of $X$ is $Coz_{\delta}$. It is enough to note that if $A=f^{-1}(0)$ for some $f\in C(X)$ then $A=\bigcap\limits_{n\in \mathbb{N}} f^{-1}((-\frac{1}{n+1},\frac{1}{n+1}))$.

It is well known that $f$ is
of the first Baire class if and only if $f^{-1}(U)\in
Zer_{\sigma}$ for every open $U\subseteq \mathbb{R}$ (see Exercise
3.A.1 in \protect\cite{lmz1}).

A $Coz_{\delta}$-subset of $X$ containing $x$ is called a {\it
$Coz_{\delta}$ neighborhood} of $x$.

\medskip

 A set $A\subseteq X$ is called {\it strongly
$Coz_{\delta}$-disjoint}, if there is a pairwise disjoint
collection $\{F_a: F_a$ is a $Coz_{\delta}$ neighborhood of $a$,
$a\in A\}$  such that $\{F_a: a\in A\}$ is a {\it completely
$Coz_{\delta}$-additive system}, i.e. $\bigcup\limits_{b\in B}
F_b\in Coz_{\delta}$ for each $B\subseteq A$.

A disjoint sequence $\{\Delta_n: n\in \mathbb{N}\}$ of (finite)
sets is  {\it strongly $Coz_{\delta}$-disjoint} if the set
$\bigcup\{\Delta_n: n\in \mathbb{N}\}$ is strongly
$Coz_{\delta}$-disjoint (see Definition 1 in \cite{Osip1}).

 In (\cite{Osip1}, see Theorem 8),  it is proved that $B_1(X)$ is Baire if and only if every pairwise disjoint sequence of non-empty finite subsets of
$X$ has a strongly $Coz_{\delta}$-disjoint subsequence.

A non-empty space is called {\it crowded} if it
has no isolated points.

A topological space $X$ is defined to be
{\it $k$-scattered} if every compact Hausdorff subspace of $X$ is scattered.

%A subset $B$ of a space $X$ is {\it $C$-compact} in $X$ if $f(B)$ is compact,
%for every real-valued continuous function $f$ on $X$.

%Call a map $f: (Y,\tau_Y)\rightarrow (Z, \tau_Z)$ is {\it almost open} if, for any $U\in \tau_Y\setminus\{\emptyset\}$, there is $V\in \tau_Z\setminus \{\emptyset\}$
%such that $f(U)\subseteq V\subseteq \overline{f(U)}$.

Denote by $C_k(X)$ the space $C(X)$ of all real-valued continuous functions on $X$ endowed with the compact-open
topology. Following \cite{BG1}, $X$ is called an {\it Ascoli space} if
every compact subset $K$ of $C_k(X)$ is evenly continuous (i.e., if the map $(f,x)\longmapsto f(x)$ is continuous
as a map from $K\times X$ to $\mathbb{R}$).

%For other notation almost without exceptions we follow the
%Engelking's book \cite{Eng}.

\section{Main results}

\begin{theorem}\label{th1} For any space $X$, the following conditions are equivalent:

\begin{enumerate}

\item  The space $C_p(X)$ is $\kappa$-Fr\'{e}chet-Urysohn.

\item The space $B_1(X)$ is Baire.

\end{enumerate}

\end{theorem}

\begin{proof} $(1)\Rightarrow(2)$.   Let $\{F_i: i\in \mathbb{N}\}$ be a pairwise disjoint sequence of non-empty finite subsets of $X$. Then $\{F_i: i\in \mathbb{N}\}$
has a strongly point-finite subsequence $\{F_{i_k}: k\in \mathbb{N}\}$, i.e.  for every $k\in \mathbb{N}$, there exists an open
set $U_{k}$ of $X$ such that $F_{i_k}\subset U_{k}$ and $\{U_{k}: k\in \mathbb{N}\}$ is point-finite. Denote by $S=\bigcup\limits_{k\in \mathbb{N}} F_{i_k}$.

Since $X$ is Tychonoff, for any pair $x,y\in S$  ($x\neq y$) there is a continuous function $f_{xy}:X\rightarrow [0,1]$ such that $f_{xy}(x)=0$ and $f_{xy}(y)=1$ and for every $k\in \mathbb{N}$ there is a continuous function $f_{k}:X\rightarrow [0,1]$ such that  $f_{k}(F_{i_k})\subseteq\{0\}$ and $f_{k}(X\setminus U_{k})\subseteq\{1\}$.

For every $k\in \mathbb{N}$ and  $x\in F_{i_k}$ denote by

$S_x=(\bigcap\{f^{-1}_{yx}(1): y\in  S\setminus \{x\}\})\cap (\bigcap\{f^{-1}_{xy}(0): y\in  S\setminus \{x\}\})\cap f^{-1}_k(0)$  and $A_k=\bigcup\{S_x: x\in F_{i_k}\}$. Note that $S_x$ is a zero-set of $X$ for every $x\in S$ because $S$ is countable.

Since $f^{-1}_{xy}(1)\cap f^{-1}_{xy}(0)=\emptyset$,
$S_x\cap S_y=\emptyset$ for any $x,y\in S$ ($x\neq y$).

Since $F_{i_k}$ is finite,  $A_k$ is a zero-set of $X$ and $F_{i_k}\subseteq A_k\subseteq U_k$ for every $k\in\mathbb{N}$.

For every $k\in\mathbb{N}$ we consider the family $\{W_{k,j}: j\in \mathbb{N}\}$ of co-zero sets of $X$ such that $A_k=\bigcap\{W_{k,j}: j\in \mathbb{N}\}$,
$W_{k,j+1}\subset W_{k,j}$ and $W_{k,1}\subset U_k$ for any $j,k\in \mathbb{N}$.

Consider the set $A=\bigcap\limits_{j\in \mathbb{N}} \bigcup\limits_{k\in \mathbb{N}} W_{k,j}$. Note that $\bigcup\limits_{k\in \mathbb{N}} A_k\subseteq A$.

\bigskip

We claim that $A\subseteq \bigcup\limits_{k\in \mathbb{N}} A_k$. Let $x\not\in \bigcup\limits_{k\in \mathbb{N}} A_k$. Then the set $\{k: x\in U_k\}$ is finite. Let $\{k_1,...,k_m\}=\{k: x\in U_k\}$. For every $s\in \{1,...,m\}$ there is $j_s\in \mathbb{N}$ such that $x\not\in W_{k_s,j_s}$. Denote by $l=\max\{j_s: s\in \{1,...,m\}\}$.
Then $x\not\in W_{k,l}$ for every $k\in \mathbb{N}$ and, hence, $x\not\in A$ and  $A\subseteq \bigcup\limits_{k\in \mathbb{N}} A_k$.

Thus $A=\bigcup\limits_{k\in \mathbb{N}} A_k$ is a $Coz_{\delta}$-set of $X$.

Let $C\subset S$. Since $\bigcup\limits_{x\in S\setminus C} S_x$ is a $Zer_{\sigma}$-set, $\bigcup\limits_{x\in C} S_x=A\setminus (\bigcup\limits_{x\in S\setminus C} S_x)$ is a $Coz_{\delta}$-set of $X$.

It follows that the family $\{S_x: x\in S\}$ of zero-sets of $X$ is disjoint and completely $\mathrm{Coz}_{\delta}$-additive.

$(2)\Rightarrow(1)$.  Let $\{F_i: i\in \mathbb{N}\}$ be a pairwise disjoint sequence of non-empty finite subsets of $X$. Then $\{F_i: i\in \mathbb{N}\}$
has a strongly $Coz_{\delta}$-disjoint subsequence $\{F_{i_k}: k\in \mathbb{N}\}$, i.e.  for every $k\in \mathbb{N}$, there exists a zero-set $S_{k}$ of $X$ such that $F_{i_k}\subset S_{k}$ and $\{S_{k}: k\in \mathbb{N}\}$ is completely $\mathrm{Coz}_{\delta}$-additive (see Theorem 1.4 in \cite{Os2}).
Let $S=\bigcup \{S_{k}: k\in \mathbb{N}\}$. Since $\{S_{k}: k\in \mathbb{N}\}$ is completely $\mathrm{Coz}_{\delta}$-additive, $S$ is a $\mathrm{Coz}_{\delta}$-set in $X$, i.e. there is a family $\{W_i: i\in \mathbb{N}\}$ of co-zero sets of $X$ such that $S=\bigcap \{W_i: i\in \mathbb{N}\}$.
For every $i\in \mathbb{N}\setminus \{1\}$, let  $H_i=W_i\setminus(\bigcup\{S_k: k\in \overline{1,i-1}\})$ and $H_1=W_1$. Then $F_{i_k}\subseteq H_i$ and $\{H_i: i\in \mathbb{N}\}$ is point-finite.

\end{proof}

A space $X$ has {\it
the Banakh property} if there is a countable family $\{A_n : n\in \mathbb{N}\}$ of closed nowhere
dense subsets of $X$ such that for any compact subset $K$ of $X$, there is $n\in \mathbb{N}$
with $K\subseteq A_n$ (see \cite{Banakh}, Definition 5.1).

In \cite{Krupski}, it is proved that $C_p(X)$ does not have the Banakh property if and only if $C_p(X)$ is $\kappa$-Fr\'{e}chet-Urysohn. In \cite{GGKL},  it is proved that the Ascoli property of $C_p(X)$ implies that $C_p(X)$ is $\kappa$-Fr\'{e}chet-Urysohn. By Theorem 2.5 in \cite{Gab}, the $\kappa$-Fr\'{e}chet-Urysohn property of $C_p(X)$ implies that $C_p(X)$ is Ascoli.

  V. Tkachuk proved (see \cite{Tkachuk}) that $C_p(X)$ is $\kappa$-Fr\'{e}chet-Urysohn if and only $C_p(X,[0,1])$ is $\kappa$-Fr\'{e}chet-Urysohn.

Combining all the previous results we get the following

\begin{corollary}\label{cor1} {\it For any space $X$, the following conditions are equivalent:

\begin{enumerate}

\item The space $X$ has the property $(\kappa)$.

\item The space $C_p(X)$ is $\kappa$-Fr\'{e}chet-Urysohn.

\item The space $C_p(X,[0,1])$ is $\kappa$-Fr\'{e}chet-Urysohn.

\item The space $C_p(X)$ is Ascoli.

\item  The space $C_p(X)$ does not have the Banakh property.

\item  Every pairwise disjoint sequence of non-empty finite subsets of
$X$ has a strongly $Coz_{\delta}$-disjoint subsequence.

\item  The space $B_1(X)$ is Baire.

\end{enumerate}}

\end{corollary}

For the remark below, we refer the reader to (\cite{Banakh}, Theorem 5.9) or (\cite{Gab}, Corollary 2.7).

\begin{remark}\label{cor11}{\it If $X$ is a compact space or, more generally, if $X$ is  \v{C}ech-complete, then all conditions in Corollary \ref{cor1} are equivalent to the condition: the space $X$ is scattered.}

\end{remark}

A topological space $X$ is called {\it almost $K$-analytic} if every countable subset of $X$ is contained in a $K$-analytic $G_{\delta}$-subspace of $X$.

By Corollary 9.8 in \cite{BG} and Theorem 17 in \cite{Osip1}, we get the following.

\begin{corollary}{\it For an almost $K$-analytic space $X$, the following conditions are equivalent:

\begin{enumerate}

\item $C_p(X)$ is $\kappa$-Fr\'{e}chet-Urysohn.

\item  $C_p(X)$ is an Ascoli space.

\item $B_1(X)$ is Choquet.

\item  Every countable subset of $X$ is strongly $Coz_{\delta}$-disjoint.

\item $X$ is $k$-scattered.

\end{enumerate}}

\end{corollary}

%\section{Examples}

By Theorem 3.19 in \cite{Tkachuk}, if $X$ is a pseudocompact first countable space with the property $(\kappa)$, then
$X$ is scattered. It was proved in \cite{Sakai} that every scattered space has the property $(\kappa)$.

Thus, we have the following

\begin{proposition}{\it For any pseudocompact first countable space $X$, the following conditions are equivalent:

\begin{enumerate}

\item  The space $B_1(X)$ is Baire.

\item  The space $C_p(X)$ does not have the Banakh property.

\item  The space $X$ is scattered.

\end{enumerate}}

\end{proposition}

By Corollary 3.3 in \cite{Tkachuk}, if all countable subsets of a space $X$ are scattered, then $X$ has the property $(\kappa)$ and, hence, by Theorem \ref{th1}, $B_1(X)$ is a Baire space.

\medskip

In \cite{JM},  Juh\'{a}sz and van Mill gave several examples of countably compact dense-in-itself spaces in which all countable subsets are scattered. Thus we get the following.

\begin{proposition} There exists a dense-in-itself countably compact space $X$ such that $B_1(X)$ is Baire.

\end{proposition}

In \cite{RT},  Reznichenko and Tkachenko showed that all countable subsets of any pseudocompact quasitopological
group in the form of a Korovin orbit are closed and $C^*$-embedded.

In (\cite{OsPy1}, Theorem 3.12), it was proved that if all countable subsets of a space $X$ are scattered and $C^*$-embedded then $C_p(X, [0,1])$ is Baire space.

Thus we get the following.

\begin{proposition}\label{pr1} There exists a pseudocompact quasitopological
group $G$ in the form of a Korovin orbit  such that $B_1(G)$ and $C_p(G,[0,1])$ are Baire spaces.

\end{proposition}

\medskip

A subset $A$ of a space $X$ is called a {\it bounded} subset (in
$X$) if every continuous real-valued function on $X$ is bounded on
$A$.

\medskip

\begin{remark}(Proposition 5.1 in \protect\cite{dkm}, Corollary 3.3 in \protect\cite{LM} for an infinite
pseudocompact subspace $A$)  If $X$ is a space containing an
infinite bounded subset $A$ (e.g. a non-trivial convergent
sequence) then $C_p(X)$ is meager. It follows that $C_p(G)$ (for the space $G$ from Proposition \ref{pr1}) is meager.
\end{remark}

\bigskip

In \cite{BG}, it is defined a space $X$ to be {\it airy} if it admits a countable family $\mathcal{P}$ of infinite subsets such that any $\mathcal{P}$-dense $G_{\delta}$-sets $G_1$, $G_2\subseteq X$ have nonempty intersection $G_1\cap G_2=\emptyset$. The $\mathcal{P}$-density of the sets $G_i$ means that $G_i\cap P=\emptyset$ for any $P\in \mathcal{P}$. Using  this notion and Theorem 1.4 in \cite{BG}, we obtain the following.

\begin{proposition}
If a space $X$ is airy, then its function space $C_p(X)$ has the Banakh property.
\end{proposition}

A topological space $X$ is defined to be  $B_{\pi}$-scattered if for any Baire space $B$ with countable $\pi$-base and any continuous map $f:B\rightarrow X$ there exists a nonempty open set $U\subseteq B$ such that $f(U)$ is finite (Definition 8.1 in \cite{BG}).

By Theorem 8.2 in \cite{BG}, if $X$ is nonaire, then $X$ is $B_{\pi}$-scattered.

\begin{corollary}
If a space $C_p(X)$ is $\kappa$-Fr\'{e}chet-Urysohn, then $X$ is $B_{\pi}$-scattered (and, hence, $k$-scattered).
\end{corollary}

In (\cite{BG}, Example 8.4), authors present a (consistent) example of a $B_{\pi}$-scattered metrizable separable space which is airy. Thus, by Theorem 1.4 in \cite{BG}, we obtain the following

\begin{proposition} $(\mathfrak{b}=\mathfrak{c})$ There exists a $B_{\pi}$-scattered  metrizable separable space $X$ such that $C_p(X)$  does not have the $\kappa$-Fr\'{e}chet-Urysohn property.
\end{proposition}

%\end{fulltext}

\bibliographystyle{model1a-num-names}
\bibliography{<your-bib-database>}

\end{document}